\documentclass{amsart}
\usepackage{amsfonts}
\usepackage{hyperref}

\newtheorem{theorem}{Theorem}
\theoremstyle{plain}

\newtheorem{conjecture}{Conjecture}
\newtheorem{corollary}{Corollary}

\newtheorem{lemma}{Lemma}

\newtheorem{property}{Property}

\newtheorem{remark}{Remark}

\numberwithin{equation}{section}
\input{tcilatex}

\begin{document}
\title[Some sharp inequalities for the Toader-Qi mean]{Some sharp
inequalities for the Toader-Qi mean}
\author{Zhen-Hang Yang}
\address{Power Supply Service Center, ZPEPC Electric Power Research
Institute, Hangzhou, Zhejiang, China, 310009}
\email{yzhkm@163.com}
\date{June 21, 2015}
\dedicatory{Dedicated to my father Xin-Jiang Yang.}
\thanks{This paper is in final form and no version of it will be submitted
for publication elsewhere.}

\begin{abstract}
The Toader-Qi mean of positive numbers $a$ and $b$ defined by%
\begin{equation*}
TQ\left( a,b\right) =\frac{2}{\pi }\int_{0}^{\pi /2}a^{\cos ^{2}\theta
}b^{\sin ^{2}\theta }d\theta
\end{equation*}%
is related to the modified Bessel function of the first kind. In this paper,
we present several properties of this mean, and establish some sharp
inequalities for this mean in terms of power and logarithmic means. From
these a nice chain of inequalities involving Gauss compound mean, Toader
mean and Toader-Qi mean is presented.
\end{abstract}

\maketitle

\section{Introduction}

Let the function $p:\left( 0,\infty \right) \rightarrow \mathbb{R}$ be
strictly monotone and let $n\in \mathbb{R}$. The Toader' family of mean
values is defined in \cite{Toader-JMAA-218(2)-1998} by%
\begin{equation*}
M_{p,n}\left( a,b\right) =p^{-1}\left( \frac{1}{2\pi }\int_{0}^{2\pi
}p\left( r_{n}\left( \theta \right) \right) d\theta \right) ,
\end{equation*}%
where%
\begin{equation*}
r_{n}\left( \theta \right) =\left\{ 
\begin{array}{ll}
\left( a^{n}\cos ^{2}\theta +b^{n}\sin ^{2}\theta \right) ^{1/n} & \text{if }%
n\neq 0, \\ 
a^{\cos ^{2}\theta }b^{\sin ^{2}\theta } & \text{if }n=0%
\end{array}%
\right.
\end{equation*}%
for $\theta \in \left( 0,2\pi \right) $, $p^{-1}$ is the inverse of a
strictly monotonic function $p$. Also, it is obvious that%
\begin{equation*}
M_{p,n}\left( a,b\right) =p^{-1}\left( \frac{1}{2\pi }\int_{0}^{2\pi
}p\left( r_{n}\left( \theta \right) \right) d\theta \right) =p^{-1}\left( 
\frac{2}{\pi }\int_{0}^{\pi /2}p\left( r_{n}\left( \theta \right) \right)
d\theta \right) .
\end{equation*}

When $p(x)=1/x$ and $n=2$, we see that%
\begin{equation}
M_{1/x,2}\left( a,b\right) =\frac{\pi /2}{\int_{0}^{\pi /2}\left( a^{2}\cos
^{2}\theta +b^{2}\sin ^{2}\theta \right) ^{-1/2}d\theta }=AGM\left(
a,b\right)  \label{AGM}
\end{equation}%
is the classical Gauss compound mean related to the complete integrals of
the first kind. Some inequalities involving $AGM$ can be found in \cite%
{Borwein-PA-1987}, \cite{Carlson-SIAMR-33-1991}, \cite%
{Borwein-JMAA-177(2)-1993}, \cite{Vamanamurthy-JMAA-183(1)-1994}, \cite%
{Sandor-JMAA-189-1995}, \cite{Neuman-IJMMS-16-2003}, \cite%
{Alzer-JCAM-172-2004}, \cite{Qi-arxiv-0902.2515V1}, \cite%
{Yang-IJMA-4(21)-2010}, \cite{Yang-JIA-192-2014}.

While letting $p(x)=x$ and $n=2$ yields%
\begin{equation}
M_{x,2}\left( a,b\right) =\frac{2}{\pi }\int_{0}^{\pi /2}\left( a^{2}\cos
^{2}\theta +b^{2}\sin ^{2}\theta \right) ^{1/2}d\theta =\mathcal{T}\left(
a,b\right) ,  \label{Toadermean}
\end{equation}%
which is the Toader mean related to the complete integrals of the second
kind. There has some papers studied bounds for the mean in terms of other
simpler means such as \cite{Alzer-JCAM-172-2004}, \cite{Qiu-JHIEE-20(1)-2000}%
, \cite{Chu-PIASMS-121-2011}, \cite{Chu-RM-61-2012}, \cite{Chu-AAA-11-2012}, 
\cite{Chu-JMI-7(2)-2013}, \cite{Song-JMI-7(4)-2013}, \cite%
{Li-JFSA-2013-394194}, \cite{Hua-F-28(4)-2014}, \cite%
{Hua-PIAS(MS)-124(4)-2014}.

Taking $p\left( x\right) =x^{q}$ ($q\neq 0$) and $n=0$ gives%
\begin{equation*}
M_{x^{q},0}\left( a,b\right) =\left( \frac{2}{\pi }\int_{0}^{\pi /2}a^{q\cos
^{2}\theta }b^{q\sin ^{2}\theta }d\theta \right) ^{1/q}.
\end{equation*}%
In particular, we have%
\begin{equation}
M_{x,0}\left( a,b\right) =\frac{2}{\pi }\int_{0}^{\pi /2}a^{\cos ^{2}\theta
}b^{\sin ^{2}\theta }d\theta .  \label{TQmean}
\end{equation}%
The mean $M_{x^{q},0}\left( a,b\right) $ seems to be mysterious so that the
author said that he did not know how to determine any mean at the end of the
paper \cite{Toader-JMAA-218(2)-1998}. Very recently, Qi et al. \cite[Lemma
2.1]{Qi-unpubl.-2015} revealed the surprising relation between the mean $%
M_{x^{q},0}\left( a,b\right) $ and modified Bessel functions of the first
kind. They proved that

\begin{theorem}[{\protect\cite[Lemma 2.1]{Qi-unpubl.-2015}}]
For positive numbers $a,b>0$, we have%
\begin{equation}
M_{x,0}\left( a,b\right) =\frac{2}{\pi }\int_{0}^{\pi /2}a^{\cos ^{2}\theta
}b^{\sin ^{2}\theta }d\theta =\sqrt{ab}I_{0}\left( \ln \sqrt{\frac{a}{b}}%
\right)  \label{Qi-r-1}
\end{equation}%
and%
\begin{equation}
M_{x^{q},0}\left( a,b\right) =\left( \frac{2}{\pi }\int_{0}^{\pi /2}a^{q\cos
^{2}\theta }b^{q\sin ^{2}\theta }d\theta \right) ^{1/q}=\sqrt{ab}%
I_{0}^{1/q}\left( q\ln \sqrt{\frac{a}{b}}\right) ,  \label{Qi-r-q}
\end{equation}%
where%
\begin{equation*}
I_{v}\left( z\right) =\sum_{n=0}^{\infty }\frac{1}{n!\Gamma \left(
v+n+1\right) }\left( \frac{z}{2}\right) ^{2n+v}\text{, \ }z\in \mathbb{C}%
\text{, \ }v\in \mathbb{R}\backslash \{-1,-2,...\}
\end{equation*}%
denotes the modified Bessel functions of the first kind and%
\begin{equation*}
\Gamma \left( z\right) =\lim_{n\rightarrow \infty }\frac{n!n^{z}}{%
\prod_{k=0}^{n}\left( z+k\right) }\text{, \ }z\in \mathbb{C}\backslash
\{-1,-2,...\}
\end{equation*}%
is the classical gamma function.
\end{theorem}

\begin{remark}
\label{R-M-H}Let $M\left( a,b\right) $ be a homogeneous mean of positive
arguments $a$ and $b$. Then 
\begin{equation*}
M\left( a,b\right) =\sqrt{ab}M\left( e^{t},e^{-t}\right) ,
\end{equation*}%
where $t=\left( 1/2\right) \ln \left( b/a\right) $.
\end{remark}

Since the mean $M_{x^{q},0}\left( a,b\right) $ is symmetric and homogeneous
with respect to $a$ and $b$, we can assume that $b>a>0$ and let $t=\left(
1/2\right) \ln \left( b/a\right) >0$. Then by Remark \ref{R-M-H} Qi's
relations (\ref{Qi-r-1}) and (\ref{Qi-r-q}) can be rewritten as%
\begin{equation}
\frac{M_{x,0}\left( a,b\right) }{\sqrt{ab}}=\frac{2}{\pi }\int_{0}^{\pi
/2}e^{t\cos 2\theta }d\theta =I_{0}\left( t\right)  \label{Q-Y-1}
\end{equation}%
and%
\begin{equation*}
\frac{M_{x^{q},0}\left( a,b\right) }{\sqrt{ab}}=\left( \frac{2}{\pi }%
\int_{0}^{\pi /2}e^{qt\cos 2\theta }d\theta \right) ^{1/q}=I_{0}^{1/q}\left(
qt\right) .
\end{equation*}

Also, from (\ref{Q-Y-1}) it is easy to verify that 
\begin{equation}
I_{0}\left( t\right) =\frac{2}{\pi }\int_{0}^{\pi /2}\cosh \left( t\cos
\theta \right) d\theta =\frac{2}{\pi }\int_{0}^{\pi /2}\cosh \left( t\sin
\theta \right) d\theta  \label{I_0-Yang-1}
\end{equation}%
(see also \cite[p. 376, 9.6.16]{Abramowitz-HMFFGMT-1972}).

Similarly, the logarithmic mean, identric (exponential) mean and power mean
of order $p$ defined by%
\begin{eqnarray*}
L\left( a,b\right)  &=&\frac{b-a}{\ln b-\ln a},\ \ \ \mathcal{I}\left(
a,b\right) =e^{-1}\left( \frac{b^{b}}{a^{a}}\right) ^{1/\left( b-a\right) },
\\
A_{p}\left( a,b\right)  &=&\left( \frac{a^{p}+b^{p}}{2}\right) ^{1/p}\text{
if }p\neq 0\text{ and }A_{0}\left( a,b\right) =\sqrt{ab}
\end{eqnarray*}%
can be rewritten as%
\begin{eqnarray*}
\frac{L\left( a,b\right) }{\sqrt{ab}} &=&\frac{\sinh t}{t}\text{, \ }\frac{%
\mathcal{I}\left( a,b\right) }{\sqrt{ab}}=\exp \left( \frac{t}{\tanh t}%
-1\right) , \\
\frac{A_{p}\left( a,b\right) }{\sqrt{ab}} &=&\cosh pt\text{ if }p\neq 0\text{
and }\frac{A_{0}\left( a,b\right) }{\sqrt{ab}}=1.
\end{eqnarray*}%
In particular, the arithmetic and geometric means $A=A_{1}\left( a,b\right) $
and $G\left( a,b\right) =A_{0}\left( a,b\right) =\sqrt{ab}$ can be changed
into $A\left( a,b\right) /\sqrt{ab}=\cosh t$ and $G\left( a,b\right) /\sqrt{%
ab}=1$, respectively.

Further, Qi et al. showed that

\begin{theorem}[{\protect\cite[Theorem 1.1]{Qi-unpubl.-2015}}]
The double inequality%
\begin{equation}
L\left( a,b\right) <\frac{2}{\pi }\int_{0}^{\pi /2}a^{\cos ^{2}\theta
}b^{\sin ^{2}\theta }d\theta <\mathcal{I}\left( a,b\right)   \label{Qi-I-1}
\end{equation}%
holds for $a,b>0$ with $a\neq b$. Consequently,%
\begin{equation*}
L\left( a^{q},b^{q}\right) ^{1/q}<\left( \frac{2}{\pi }\int_{0}^{\pi
/2}a^{q\cos ^{2}\theta }b^{q\sin ^{2}\theta }d\theta \right) ^{1/q}<\mathcal{%
I}\left( a^{q},b^{q}\right) ^{1/q}.
\end{equation*}
\end{theorem}

\begin{remark}
Due to that Qi et al. first revealed the surprising connection between the
mean $M_{x,0}\left( a,b\right) $ and the modified Bessel functions of the
first kind, and established inequalities for the mean in terms of
logarithmic and identric (exponential) means, we call the mean $%
M_{x,0}\left( a,b\right) $ defined by (\ref{TQmean}) Toader-Qi mean and
denote by $TQ\left( a,b\right) $.
\end{remark}

At the end of paper \cite{Qi-unpubl.-2015}, Qi et al. gave some improvements
for the second inequality in (\ref{Qi-I-1}), that is,%
\begin{equation}
\frac{2}{\pi }\int_{0}^{\pi /2}a^{\cos ^{2}\theta }b^{\sin ^{2}\theta
}d\theta <\frac{A\left( a,b\right) +G\left( a,b\right) }{2}<\frac{2A\left(
a,b\right) +G\left( a,b\right) }{3}<\mathcal{I}\left( a,b\right) .
\label{Qi-I-2}
\end{equation}

The aim of this paper is to present some sharp inequalities for the
Toader-Qi mean $TQ\left( a,b\right) $, or equivalently, modified Bessel
functions of the first kind $I_{0}\left( t\right) $ in terms of hyperbolic
functions.

\section{Lemmas}

To formulate properties of the Toader-Qi mean $TQ\left( a,b\right) $ or $%
I_{0}\left( t\right) $ and our results, we need some lemmas.

\begin{lemma}[{\protect\cite[Problem 32]{Polya-PTA-I-CM-SV-1998}}]
\label{L-C(n,k)^2}Let $\dbinom{n}{k}$ be the number of combinations of $n$
objects taken $k$ at a time, that is,%
\begin{equation*}
\dbinom{n}{k}=\frac{n!}{k!(n-k)!}.
\end{equation*}%
Then we have%
\begin{equation*}
\sum_{k=0}^{n}\dbinom{n}{k}^{2}=\dbinom{2n}{n}.
\end{equation*}
\end{lemma}

\begin{lemma}[{\protect\cite[Problems 85, 94]{Polya-PTA-I-CM-SV-1998}}]
\label{L-lima_n/b_n}The two given sequences $\{a_{n}\}_{n\geq 0}$ and $%
\{b_{n}\}_{n\geq 0}$ satisfy the conditions%
\begin{equation*}
b_{n}>0\text{; \ }\sum_{n=0}^{\infty }b_{n}t^{n}\text{ converges for all
values of }t\text{; \ }\lim_{n\rightarrow \infty }\frac{a_{n}}{b_{n}}=s.
\end{equation*}%
Then $\sum_{n=0}^{\infty }a_{n}t^{n}$ converges too for all values of $t$
and in addition%
\begin{equation*}
\lim_{t\rightarrow \infty }\frac{\sum_{n=0}^{\infty }a_{n}t^{n}}{%
\sum_{n=0}^{\infty }b_{n}t^{n}}=s.
\end{equation*}
\end{lemma}

\begin{lemma}
\label{L-W_n}The Wallis ratio $W_{n}$ defined by%
\begin{equation}
W_{n}=\frac{\left( 2n-1\right) !!}{\left( 2n\right) !!}=\frac{\left(
2n\right) !}{2^{2n}n!^{2}}=\frac{\Gamma \left( n+1/2\right) }{\Gamma \left(
1/2\right) \Gamma \left( n+1\right) }  \label{W_n}
\end{equation}%
is strictly decreasing and log-convex for all integers $n\geq 0$.
\end{lemma}

\begin{proof}
Direct computations gives%
\begin{eqnarray}
\frac{W_{n+1}}{W_{n}} &=&\frac{2n+1}{2n+2}=1-\frac{1}{2n+2}<1,
\label{W_n+1-W_n} \\
\left. \frac{W_{n+2}}{W_{n+1}}\right/ \frac{W_{n+1}}{W_{n}} &=&\frac{2n+3}{%
2n+4}\frac{2n+2}{2n+1}=1+\frac{1}{\left( n+2\right) \left( 2n+1\right) }>1, 
\notag
\end{eqnarray}%
which proves the lemma.
\end{proof}

\begin{lemma}[{\protect\cite{Wendel-AMM-55-1948}, \protect\cite[(2.8)]%
{Qi-JIA-2010-493058}}]
\label{L-g(x+a)/g(x+1)>}For all $x>0$ and all $a\in \left( 0,1\right) $, it
holds that%
\begin{equation}
\frac{1}{\left( x+a\right) ^{1-a}}<\frac{\Gamma \left( x+a\right) }{\Gamma
\left( x+1\right) }<\frac{1}{x^{1-a}}.  \label{L-gi><}
\end{equation}
\end{lemma}

\begin{lemma}
\label{L-s_n}The sequence $\{s_{n}\}_{n\geq 0}$ defined by%
\begin{equation}
s_{n}=\frac{\left( 2n\right) !\left( 2n+1\right) !}{2^{4n}n!^{4}}
\label{s_n}
\end{equation}%
is strictly decreasing, and $\lim_{n\rightarrow \infty }s_{n}=2/\pi $.
\end{lemma}

\begin{proof}
An easy computation yields%
\begin{equation}
\frac{s_{n+1}}{s_{n}}=\left. \frac{\left( 2n+2\right) !\left( 2n+3\right) !}{%
2^{4n+4}\left( n+1\right) !^{4}}\right/ \frac{\left( 2n\right) !\left(
2n+1\right) !}{2^{4n}n!^{4}}=\frac{1}{4}\frac{\left( 2n+1\right) \left(
2n+3\right) }{\left( n+1\right) ^{2}}<1,  \label{s_n-rec-rel}
\end{equation}%
which shows that the sequence $\{s_{n}\}$ is strictly decreasing for all $%
n\geq 0$. To calculate $\lim_{n\rightarrow \infty }s_{n}$, we write $s_{n}$
as%
\begin{equation*}
s_{n}=\frac{\left( 2n\right) !\left( 2n+1\right) !}{2^{4n}n!^{4}}=\left(
2n+1\right) \left( \frac{\left( 2n-1\right) !!}{2^{n}n!}\right) ^{2}=\frac{%
2n+1}{\Gamma \left( 1/2\right) ^{2}}\left( \frac{\Gamma \left( n+1/2\right) 
}{\Gamma \left( n+1\right) }\right) ^{2},
\end{equation*}%
and it suffices to prove that%
\begin{equation*}
\lim_{n\rightarrow \infty }\frac{\left( n+1/2\right) \Gamma \left(
n+1/2\right) ^{2}}{\Gamma \left( n+1\right) ^{2}}=1.
\end{equation*}%
Making use of Lemma \ref{L-g(x+a)/g(x+1)>} yields%
\begin{equation*}
1=\frac{n+1/2}{n+1/2}<\frac{\left( n+1/2\right) \Gamma \left( n+1/2\right)
^{2}}{\Gamma \left( n+1\right) ^{2}}<\frac{n+1/2}{n},
\end{equation*}%
which implies the desired assertion.
\end{proof}

\begin{lemma}[\protect\cite{Biernacki-AUMC-S-9-1955}]
\label{L-A/B}Let $A\left( t\right) =\sum_{k=0}^{\infty }a_{k}t^{k}$ and $%
B\left( t\right) =\sum_{k=0}^{\infty }b_{k}t^{k}$ be two real power series
converging on $\left( -r,r\right) $ ($r>0$) with $b_{k}>0$ for all $k$. If
the sequence $\{a_{k}/b_{k}\}\ $is increasing (decreasing) for all $k$, then
the function $t\mapsto A\left( t\right) /B\left( t\right) $ is also
increasing (decreasing) on $\left( 0,r\right) $.
\end{lemma}

\begin{lemma}[{\protect\cite[Corollary 2.3.]{Yang-JMAA-428-2015}}]
\label{L-A/B-g}Let $A\left( t\right) =\sum_{k=0}^{\infty }a_{k}t^{k}$ and $%
B\left( t\right) =\sum_{k=0}^{\infty }b_{k}t^{k}$ be two real power series
converging on $\mathbb{R}$ with $b_{k}>0$ for all $k$. If for certain $m\in 
\mathbb{N}$, the non-constant sequence $\{a_{k}/b_{k}\}$ is increasing
(decreasing) for $0\leq k\leq m$ and decreasing (increasing) for $k\geq m$,
then there is a unique $t_{0}\in \left( 0,\infty \right) $ such that the
function $A/B$ is increasing (decreasing) on $\left( 0,t_{0}\right) $ and
decreasing (increasing) on $\left( t_{0},\infty \right) $.
\end{lemma}

\section{Properties}

Now we give some simple properties of the Toader-Qi mean $TQ\left(
a,b\right) $ or $I_{0}\left( t\right) $.

By the identities (\ref{I_0-Yang-1}), the following property is immediate.

\begin{property}
\label{P-G-TQ-A}For $t>0$, it holds that%
\begin{equation*}
1<I_{0}\left( t\right) <\cosh t,
\end{equation*}%
or equivalently, the double inequality%
\begin{equation}
\sqrt{ab}<TQ\left( a,b\right) <\frac{a+b}{2}  \label{G-TQ-A}
\end{equation}%
holds for $a,b>0$ with $a\neq b$.
\end{property}

Making a change of variable $\sin \theta =x$ in the second identity of (\ref%
{I_0-Yang-1}) yields

\begin{property}
We have%
\begin{equation}
I_{0}\left( t\right) =\frac{2}{\pi }\int_{0}^{1}\frac{\cosh \left( tx\right) 
}{\sqrt{1-x^{2}}}dx  \label{I_0-Yang-2}
\end{equation}%
(see also \cite[p. 376, 9.6.18]{Abramowitz-HMFFGMT-1972}).
\end{property}

\begin{property}
\label{P-I<e^t/s}For $t>0$, it holds that%
\begin{equation}
\frac{e^{t}}{1+2t}<I_{0}\left( t\right) <\frac{e^{t}}{\sqrt{1+2t}},
\label{I-e^t}
\end{equation}%
or equivalently,%
\begin{equation}
\frac{b}{1+\ln \left( b/a\right) }<TQ\left( a,b\right) <\frac{b}{\sqrt{1+\ln
\left( b/a\right) }}  \label{TQ-b}
\end{equation}%
holds for $b>a>0$. Consequently, we have%
\begin{equation}
\lim_{t\rightarrow \infty }e^{-t}I_{0}\left( t\right) =0  \label{L=0}
\end{equation}%
or%
\begin{equation}
\lim_{x\rightarrow 0^{+}}TQ\left( x,1\right) =0.  \label{TQ(0+,1)=0}
\end{equation}
\end{property}

\begin{proof}
An easy verification shows that both the functions $1/\sqrt{1-x^{2}}$ and $%
\cosh \left( tx\right) $ are increasing with respect to $x$ on $\left[ 0,1%
\right] $. Using the Chebyshev integral inequality to the formula (\ref%
{I_0-Yang-2}) we get%
\begin{equation*}
I_{0}\left( t\right) >\frac{2}{\pi }\int_{0}^{1}\frac{dx}{\sqrt{1-x^{2}}}%
\int_{0}^{1}\cosh \left( tx\right) dx=\frac{\sinh t}{t}=\frac{e^{t}}{2t}%
\left( 1-\frac{1}{e^{2t}}\right) >\frac{e^{t}}{2t+1},
\end{equation*}%
where the last inequality holds due to $e^{2t}>1+2t$, which proves the first
inequality in (\ref{I-e^t}).

On the other hand, we have%
\begin{eqnarray*}
e^{-t}I_{0}\left( t\right) &=&\frac{2}{\pi }\int_{0}^{\pi /2}e^{-t}e^{t\cos
2\theta }d\theta =\frac{2}{\pi }\int_{0}^{\pi /2}\frac{d\theta }{e^{2t\sin
^{2}\theta }}<\int_{0}^{\pi /2}\frac{d\theta }{1+2t\sin ^{2}\theta } \\
&=&\frac{2}{\pi }\left[ \frac{\arctan \left( \sqrt{1+2t}\tan \theta \right) 
}{\sqrt{1+2t}}\right] _{\theta =0}^{\theta =\pi /2}=\frac{1}{\sqrt{1+2t}},
\end{eqnarray*}%
which proves the second inequality in (\ref{I-e^t}). From the double
inequality (\ref{I-e^t}) the limit relation (\ref{L=0}) easily follows.

Substituting $t=\left( 1/2\right) \ln \left( b/a\right) $ $\left(
b>a>0\right) $ into (\ref{I-e^t}) and (\ref{L=0}) gives (\ref{TQ-b}) and (%
\ref{TQ(0+,1)=0}).
\end{proof}

\begin{remark}
By the limit relation (\ref{TQ(0+,1)=0}) and homogeneous of $TQ\left(
a,b\right) $ with respect to positive numbers $a$ and $b$, we see that the
Toader-Qi mean $TQ\left( a,b\right) $ can be extended continuously to the
domain $\{\left( a,b\right) |a,b\geq 0\}$.
\end{remark}

\begin{property}
\label{P-I_0^2}We have%
\begin{equation}
I_{0}\left( t\right) ^{2}=\sum_{n=0}^{\infty }\frac{\left( 2n\right) !}{%
2^{2n}n!^{4}}t^{2n}.  \label{I_0^2}
\end{equation}
\end{property}

\begin{proof}
By Cauchy product formula and Lemma \ref{L-C(n,k)^2}, it is obtained that%
\begin{eqnarray*}
I_{0}\left( t\right) ^{2} &=&\sum_{n=0}^{\infty }\left( \sum_{k=0}^{n}\frac{1%
}{2^{2k}k!^{2}}\frac{1}{2^{2\left( n-k\right) }\left( n-k\right) !^{2}}%
\right) t^{2n} \\
&=&\sum_{n=0}^{\infty }\left( \frac{1}{2^{2n}n!^{2}}\sum_{k=0}^{n}\frac{%
n!^{2}}{k!^{2}\left( n-k\right) !^{2}}\right) t^{2n}=\sum_{n=0}^{\infty }%
\frac{\left( 2n\right) !}{2^{2n}n!^{4}}t^{2n}.
\end{eqnarray*}
\end{proof}

\begin{property}
\label{P-Tq-L_2}The function%
\begin{equation*}
t\mapsto I_{0}\left( t\right) \left/ \sqrt{\frac{\sinh 2t}{2t}}\right. 
\end{equation*}%
is strictly decreasing from $\left( 0,\infty \right) $ onto $\left( \sqrt{%
2/\pi },1\right) $. Consequently, the double inequality%
\begin{equation}
\sqrt{\frac{\sinh 2t}{\pi t}}<I_{0}\left( t\right) <\sqrt{\frac{\sinh 2t}{2t}%
}  \label{I_0-sh2t/2t}
\end{equation}%
holds for $t>0$, or equivalently, the double inequality 
\begin{equation}
\sqrt{\frac{2}{\pi }}\sqrt{L\left( a,b\right) A\left( a,b\right) }<TQ\left(
a,b\right) <\sqrt{L\left( a,b\right) A\left( a,b\right) }
\label{I_0-sqr(LA)}
\end{equation}%
holds for $a,b>0$ with $a\neq b$, where $\sqrt{2/\pi }$ and $1$ are the best
possible.
\end{property}

\begin{proof}
Using the identity (\ref{I_0^2}) we have%
\begin{equation*}
R_{0}\left( t\right) :=\frac{I_{0}\left( t\right) ^{2}}{\left( \sinh
2t\right) /\left( 2t\right) }=\frac{\sum_{n=0}^{\infty }\frac{\left(
2n\right) !}{2^{2n}n!^{4}}t^{2n}}{\sum_{n=0}^{\infty }\frac{2^{2n}}{\left(
2n+1\right) !}t^{2n}}:=\frac{\sum_{n=0}^{\infty }a_{n}t^{2n}}{%
\sum_{n=0}^{\infty }b_{n}t^{2n}}.
\end{equation*}%
It is obvious that%
\begin{equation*}
\frac{a_{n}}{b_{n}}=\left. \frac{\left( 2n\right) !}{2^{2n}n!^{4}}\right/ 
\frac{2^{2n}}{\left( 2n+1\right) !}=\frac{\left( 2n\right) !\left(
2n+1\right) !}{2^{4n}n!^{4}}=s_{n}.
\end{equation*}%
By Lemma \ref{L-s_n} it follows that the sequence $\{a_{n}/b_{n}\}$ is
strictly decreasing for all integers $n\geq 0$, so is the function $R_{0}$
on $\left( 0,\infty \right) $ by Lemma \ref{L-A/B}. Consequently, it is
obtained that%
\begin{equation*}
\frac{2}{\pi }=\lim_{t\rightarrow \infty }R_{0}\left( t\right) <R_{0}\left(
t\right) =\lim_{t\rightarrow 0}R_{0}\left( t\right) =1,
\end{equation*}%
where the first equality holds due to%
\begin{equation}
\lim_{t\rightarrow \infty }R_{0}\left( t\right) =\lim_{t\rightarrow \infty }%
\frac{\sum_{n=0}^{\infty }a_{n}t^{2n}}{\sum_{n=0}^{\infty }b_{n}t^{2n}}%
=\lim_{n\rightarrow \infty }\frac{a_{n}}{b_{n}}=\lim_{n\rightarrow \infty }%
\frac{\left( 2n\right) !\left( 2n+1\right) !}{2^{4n}n!^{4}}=\frac{2}{\pi }
\label{limI_0-sh2t/2t}
\end{equation}%
by Lemmas \ref{L-lima_n/b_n} and \ref{L-s_n}.

This completes the proof.
\end{proof}

\begin{remark}
The limit relation \ref{limI_0-sh2t/2t} implies that%
\begin{equation*}
\lim_{t\rightarrow \infty }\sqrt{t}e^{-t}I_{0}\left( t\right) =\frac{1}{%
\sqrt{2\pi }}\text{ or }I_{0}\left( t\right) \sim \frac{e^{t}}{\sqrt{2\pi t}}%
\text{ as }t\rightarrow \infty
\end{equation*}%
(see also \cite[p. 377, 9.7.1]{Abramowitz-HMFFGMT-1972}).
\end{remark}

\section{Main results}

Due to Remark \ref{R-M-H}, almost all of inequalities for homogeneous
symmetric bivariate means can be transformed equivalently into the
corresponding ones for hyperbolic functions and vice versa, for example, the
double inequalities (\ref{I_0-sh2t/2t}) and (\ref{I_0-sqr(LA)}) are
equivalent to each other. Therefore, for convenience, we only present sharp
inequalities for the modified Bessel functions of the first kind $%
I_{0}\left( t\right) $ in terms of hyperbolic functions in this section.

\begin{theorem}
\label{MT-TQ-sqrLA-w}The double inequality%
\begin{equation}
\sqrt{\left( \lambda \cosh t+1-\lambda \right) \frac{\sinh t}{t}}%
<I_{0}\left( t\right) <\sqrt{\left( \delta \cosh t+1-\delta \right) \frac{%
\sinh t}{t}}  \label{I-TQ-sqrLA-w}
\end{equation}%
holds for all $t>0$ if and only $\lambda \in \left[ 0,2/\pi \right] $ and $%
\delta \in \lbrack \delta _{0},\infty )$, where $\delta _{0}\approx 0.67664$
is defined by%
\begin{equation*}
\delta _{0}=\frac{t_{0}I_{0}\left( t_{0}\right) ^{2}-\sinh t_{0}}{\left(
\cosh t_{0}-1\right) \sinh t_{0}},
\end{equation*}%
here $t_{0}$ is the unique solution of the equation%
\begin{equation*}
\frac{d}{dt}\left( \frac{tI_{0}\left( t\right) ^{2}-\sinh t}{\left( \cosh
t-1\right) \sinh t}\right) =0
\end{equation*}%
on $\left( 0,\infty \right) $.
\end{theorem}

\begin{proof}
Let us consider the ratio%
\begin{equation*}
R_{1}\left( t\right) =\frac{I_{0}\left( t\right) ^{2}-\left( \sinh t\right)
/t}{\left( \cosh t-1\right) \left( \sinh t\right) /t}=\frac{%
\sum_{n=1}^{\infty }\left( \frac{\left( 2n\right) !}{2^{2n}n!^{4}}-\frac{1}{%
\left( 2n+1\right) !}\right) t^{2n}}{\sum_{n=1}^{\infty }\frac{2^{2n}-1}{%
\left( 2n+1\right) !}t^{2n}}:=\frac{\sum_{n=1}^{\infty }c_{n}t^{2n}}{%
\sum_{n=1}^{\infty }d_{n}t^{2n}}.
\end{equation*}%
To determine the monotonicity of $R_{2}\left( t\right) $, it suffices to
observe the monotonicity of the sequence $\{c_{n}/d_{n}\}$. We have%
\begin{eqnarray*}
\frac{c_{n}}{d_{n}} &=&\left. \left( \frac{\left( 2n\right) !}{2^{2n}n!^{4}}-%
\frac{1}{\left( 2n+1\right) !}\right) \right/ \frac{2^{2n}-1}{\left(
2n+1\right) !} \\
&=&\left. \left( 2^{2n}\times \frac{\left( 2n\right) !\left( 2n+1\right) !}{%
2^{4n}n!^{4}}-1\right) \right/ \left( 2^{2n}-1\right) =\frac{2^{2n}s_{n}-1}{%
2^{2n}-1},
\end{eqnarray*}%
where $s_{n}$ is defined by (\ref{s_n}). Then it is obtained by the relation
(\ref{s_n-rec-rel}) that%
\begin{eqnarray*}
\frac{c_{n+1}}{d_{n+1}}-\frac{c_{n}}{d_{n}} &=&\frac{2^{2n+2}s_{n+1}-1}{%
2^{2n+2}-1}-\frac{2^{2n}s_{n}-1}{2^{2n}-1} \\
&=&\frac{2^{2n+2}\frac{\left( 2n+1\right) \left( 2n+3\right) }{4\left(
n+1\right) ^{2}}s_{n}-1}{2^{2n+2}-1}-\frac{2^{2n}s_{n}-1}{2^{2n}-1} \\
&=&-\frac{2^{2n}\times s_{n}^{\prime }}{\left( n+1\right) ^{2}\left(
2^{2n+2}-1\right) \left( 2^{2n}-1\right) },
\end{eqnarray*}%
where%
\begin{equation*}
s_{n}^{\prime }=\left( 2^{2n}+3n^{2}+6n+2\right) s_{n}-\left(
3n^{2}+6n+3\right) .
\end{equation*}%
Since the sequence $\{s_{n}\}$ is strictly decreasing for $n\geq 0$ and $%
\lim_{n\rightarrow \infty }s_{n}=2/\pi >3/5$, we get%
\begin{eqnarray*}
s_{n}^{\prime } &>&\left( 2^{2n}+3n^{2}+6n+2\right) \frac{3}{5}-\left(
3n^{2}+6n+3\right) \\
&=&\frac{3}{5}\left( 2^{2n}-\left( 2n^{2}+4n+3\right) \right) >0
\end{eqnarray*}%
for $n\geq 3$. Therefore, the sequence $\{c_{n}/d_{n}\}$ is strictly
decreasing for $n\geq 3$.

On the other hand, a direct computation yields%
\begin{equation*}
\frac{c_{1}}{d_{1}}=\frac{2}{3}<\frac{c_{2}}{d_{2}}=\frac{41}{60}>\frac{c_{3}%
}{d_{3}}=\frac{19}{28}.
\end{equation*}%
These shows that the sequence $\{c_{n}/d_{n}\}$ is strictly increasing for $%
n=1,2$ and decreasing for $n\geq 2$. By Lemma \ref{L-A/B-g}, there is a
unique $t_{0}\in \left( 0,\infty \right) $ such that the function $R_{2}$ is
strictly increasing on $\left( 0,t_{0}\right) $ and decreasing on $\left(
t_{0},\infty \right) $. Therefore, we conclude that%
\begin{equation*}
\frac{2}{\pi }=\min \left( R_{1}\left( 0^{+}\right) ,R_{1}\left( \infty
\right) \right) <R_{1}\left( t\right) \leq R_{1}\left( t_{0}\right) =\delta
_{0},
\end{equation*}%
where the first equality holds due to $R_{1}\left( 0^{+}\right) =2/3$ and%
\begin{equation*}
R_{1}\left( \infty \right) =\lim_{t\rightarrow \infty }\frac{%
\sum_{n=1}^{\infty }c_{n}t^{2n}}{\sum_{n=1}^{\infty }d_{n}t^{2n}}%
=\lim_{n\rightarrow \infty }\frac{c_{n}}{d_{n}}=\lim_{n\rightarrow \infty }%
\frac{2^{2n}s_{n}-1}{2^{2n}-1}=\frac{2}{\pi }
\end{equation*}%
by Lemmas \ref{L-lima_n/b_n} and \ref{L-s_n}. Solving the equation $%
R_{1}^{\prime }\left( t\right) =0$, we find that $t_{0}\approx 2.7113555314$%
, and $R_{1}\left( t_{0}\right) \approx 0.67664$.

Thus we complete the proof.
\end{proof}

\begin{theorem}
\label{MT-TQ-LA}Let $p,q\in \mathbb{R}$. The double inequality%
\begin{equation}
\left( \cosh t\right) ^{1-p}\left( \frac{\sinh t}{t}\right) ^{p}<I_{0}\left(
t\right) <q\frac{\sinh t}{t}+\left( 1-q\right) \cosh t  \label{I_0-L-A}
\end{equation}%
holds for $t>0$ if and only if $p\geq 3/4$ and $q\leq 3/4$.
\end{theorem}

\begin{proof}
(i) The necessity of the first inequality in (\ref{I_0-L-A}) follows from
the expansion in power series%
\begin{equation*}
I_{0}\left( t\right) -\left( \cosh t\right) ^{1-p}\left( \frac{\sinh t}{t}%
\right) ^{p}=\frac{1}{3}t^{2}\left( p-\frac{3}{4}\right) +O\left(
t^{4}\right) .
\end{equation*}%
Since the function $p\mapsto \left( \cosh t\right) ^{1-p}\left( \left( \sinh
t\right) /t\right) ^{p}$ is decreasing, to prove the sufficiency, it is
enough to prove that the first inequality in (\ref{I_0-L-A}) holds for $p=3/4
$, that is, 
\begin{equation*}
I_{0}\left( t\right) >\left( \cosh t\right) ^{1/4}\left( \frac{\sinh t}{t}%
\right) ^{3/4},
\end{equation*}%
which is equivalent to%
\begin{equation*}
I_{0}\left( t\right) ^{4}>\left( \frac{\sinh t}{t}\right) ^{3}\cosh t.
\end{equation*}%
Expanding in power series yields%
\begin{equation*}
\left( \cosh t\right) \left( \frac{\sinh t}{t}\right)
^{3}=\sum_{n=0}^{\infty }\frac{2^{4n+3}-2^{2n+1}}{\left( 2n+3\right) !}%
t^{2n}.
\end{equation*}%
On the other hand, by Cauchy product formula and formula (\ref{I_0^2}), it
is obtained that%
\begin{equation*}
I_{0}\left( t\right) ^{4}=\sum_{n=0}^{\infty }\sum_{k=0}^{n}\left( \frac{%
\left( 2k\right) !}{2^{2k}k!^{4}}\frac{\left( 2\left( n-k\right) \right) !}{%
2^{2\left( n-k\right) }\left( n-k\right) !^{4}}\right)
t^{2n}:=\sum_{n=0}^{\infty }\sum_{k=0}^{n}u_{n,k}t^{2n}.
\end{equation*}

Thus it suffices to prove that%
\begin{equation*}
v_{n}=\sum_{k=0}^{n}u_{n,k}-\frac{2^{4n+3}-2^{2n+1}}{\left( 2n+3\right) !}%
\geq 0
\end{equation*}%
for $n\geq 0$. To this end, we use the identity (\ref{W_n}) to $u_{n,k}$,
then apply Lemma \ref{L-W_n} and inequality (\ref{L-gi><}), to get that%
\begin{eqnarray*}
u_{n,k} &=&\frac{1}{k!^{2}\left( n-k\right) !^{2}}\frac{\left( 2k\right) !}{%
2^{2k}k!^{2}}\frac{\left( 2\left( n-k\right) \right) !}{2^{2\left(
n-k\right) }\left( n-k\right) !^{2}} \\
&=&\frac{1}{k!^{2}\left( n-k\right) !^{2}}W_{k}W_{n-k} \\
&>&\frac{1}{k!^{2}\left( n-k\right) !^{2}}W_{n/2}^{2}=\frac{1}{k!^{2}\left(
n-k\right) !^{2}}\left( \frac{\Gamma \left( n/2+1/2\right) }{\Gamma \left(
1/2\right) \Gamma \left( n/2+1\right) }\right) ^{2} \\
&>&\frac{1}{\pi }\frac{1}{k!^{2}\left( n-k\right) !^{2}}\frac{1}{\left(
n/2+1/2\right) }.
\end{eqnarray*}%
Then it follows from Lemma \ref{L-C(n,k)^2} that%
\begin{eqnarray*}
\sum_{k=0}^{n}u_{n,k} &>&\sum_{k=0}^{n}\frac{1}{\pi }\frac{1}{k!^{2}\left(
n-k\right) !^{2}}\frac{1}{\left( n/2+1/2\right) } \\
&=&\frac{2}{\pi \left( n+1\right) n!^{2}}\sum_{k=0}^{n}\frac{n!^{2}}{%
k!^{2}\left( n-k\right) !^{2}}=\frac{2}{\pi \left( n+1\right) n!^{2}}\frac{%
\left( 2n\right) !}{n!^{2}},
\end{eqnarray*}%
and then,%
\begin{eqnarray*}
v_{n} &=&\sum_{k=0}^{n}u_{n,k}-\frac{2^{4n+3}-2^{2n+1}}{\left( 2n+3\right) !}%
>\frac{2}{\pi \left( n+1\right) n!^{2}}\frac{\left( 2n\right) !}{n!^{2}}-%
\frac{2^{4n+3}-2^{2n+1}}{\left( 2n+3\right) !} \\
&=&\frac{1}{\pi }\frac{2^{2n+1}\left( 2^{2n+2}-1\right) }{\left( 2n+3\right)
!}\left( \frac{2^{2n+1}\left( 2n+3\right) }{2^{2n+2}-1}\frac{\left(
2n\right) !\left( 2n+1\right) !}{2^{4n}n!^{4}}-\pi \right)  \\
&=&\frac{1}{\pi }\frac{2^{2n+1}\left( 2^{2n+2}-1\right) }{\left( 2n+3\right)
!}\left( \frac{2^{2n+2}}{2^{2n+2}-1}\left( n+\frac{3}{2}\right) s_{n}-\pi
\right) ,
\end{eqnarray*}%
where $s_{n}$ is defined by (\ref{s_n}).

By Lemma \ref{L-s_n} it follows that%
\begin{equation*}
\frac{2^{2n+2}}{2^{2n+2}-1}\left( n+\frac{3}{2}\right) s_{n}-\pi >\left( n+%
\frac{3}{2}\right) \frac{2}{\pi }-\pi >0
\end{equation*}%
for $n\geq 4$, which implies that $v_{n}>0$ for $n\geq 4$.

This together with the facts that $v_{0}=v_{1}=0$, $v_{2}=3/80$, $%
v_{3}=4/189 $ indicates that $v_{n}\geq 0$ for all integers $n\geq 0$, which
proves the sufficiency.

(ii) The necessity of the second inequality in (\ref{I_0-L-A}) can be
derived from the expansion in power series 
\begin{equation*}
I_{0}\left( t\right) -q\frac{\sinh t}{t}-\left( 1-q\right) \cosh t=\frac{1}{3%
}t^{2}\left( q-\frac{3}{4}\right) +O\left( t^{4}\right) .
\end{equation*}

To prove the sufficiency, let us consider the ratio%
\begin{eqnarray*}
R_{2}\left( t\right)  &=&\frac{\cosh t-I_{0}\left( t\right) }{\cosh t-\left(
\sinh t\right) /t}=\frac{\sum_{n=0}^{\infty }\frac{t^{2n}}{\left( 2n\right) !%
}-\sum_{n=0}^{\infty }\frac{t^{2n}}{2^{2n}n!^{2}}}{\sum_{n=0}^{\infty }\frac{%
t^{2n}}{\left( 2n\right) !}-\sum_{n=0}^{\infty }\frac{t^{2n}}{\left(
2n+1\right) !}} \\
&=&\frac{\sum_{n=0}^{\infty }\frac{2^{n}n!-\left( 2n-1\right) !!}{%
2^{n}n!\left( 2n\right) !}t^{2n}}{\sum_{n=0}^{\infty }\frac{2n}{\left(
2n+1\right) !}t^{2n}}:=\frac{\sum_{n=1}^{\infty }\alpha _{n}t^{2n}}{%
\sum_{n=1}^{\infty }\beta _{n}t^{2n}}.
\end{eqnarray*}%
Simplifying yields%
\begin{equation*}
\frac{\alpha _{n}}{\beta _{n}}=\frac{\left( 2n+1\right) \left(
2^{n}n!-\left( 2n-1\right) !!\right) }{n2^{n+1}n!}=\frac{2n+1}{2n}\left(
1-W_{n}\right) ,
\end{equation*}%
where $W_{n}$ is defined by (\ref{W_n}). Using the recursive relation (\ref%
{W_n+1-W_n}) we have%
\begin{eqnarray*}
\frac{\alpha _{n+1}}{\beta _{n+1}}-\frac{\alpha _{n}}{\beta _{n}} &=&\frac{%
2n+3}{2n+2}\left( 1-W_{n+1}\right) -\frac{2n+1}{2n}\left( 1-W_{n}\right)  \\
&=&\frac{2n+3}{2n+2}\left( 1-\frac{2n+1}{2n+2}W_{n}\right) -\frac{2n+1}{2n}%
\left( 1-W_{n}\right)  \\
&=&\frac{1}{2n\left( n+1\right) }\left( \frac{\left( n+2\right) \left(
2n+1\right) }{2\left( n+1\right) }W_{n}-1\right) :=\frac{\gamma _{n}-1}{%
2n\left( n+1\right) }.
\end{eqnarray*}

Since%
\begin{equation*}
\frac{\gamma _{n+1}}{\gamma _{n}}=\frac{\left( 2n+3\right) \left( n+3\right) 
}{2\left( n+2\right) ^{2}}=1+\frac{n+1}{2\left( n+2\right) ^{2}}>1,
\end{equation*}%
the sequence $\{\gamma _{n}\}$ is strictly increasing for $n\geq 1$, we have 
$\gamma _{n}\geq \gamma _{1}=9/8>1$. This in turn implies that the sequence $%
\{\alpha _{n}/\beta _{n}\}$ is strictly increasing for $n\geq 1$, so that
the ratio $R_{1}$ is increasing for $t>0$. Thus we conclude that%
\begin{equation*}
R_{2}\left( t\right) =\frac{\cosh t-I_{0}\left( t\right) }{\cosh t-\left(
\sinh t\right) /t}>\lim_{t\rightarrow 0^{+}}\frac{\cosh t-I_{0}\left(
t\right) }{\cosh t-\left( \sinh t\right) /t}=\frac{3}{4},
\end{equation*}%
which proves the sufficiency.

This completes the proof.
\end{proof}

\begin{theorem}
\label{MT-TQ-A_p+G}Let $p\in \left( 0,\infty \right) $. Then (i) the double
inequality%
\begin{equation}
1-\frac{1}{2p^{2}}+\frac{1}{2p^{2}}\cosh pt<I_{0}\left( t\right) <1-\frac{1}{%
2q^{2}}+\frac{1}{2q^{2}}\cosh qt  \label{I_0-chpt}
\end{equation}%
holds for $t>0$ if and only if $p\in \left( 0,\sqrt{3}/2\right] $ and $q\in %
\left[ 1,\infty \right) $.

(ii) For $p\in \left( \sqrt{3}/2,1\right) $, the inequality$\,$%
\begin{equation}
I_{0}\left( t\right) \geq 1-\frac{\lambda _{0}}{p^{2}}+\frac{\lambda _{0}}{%
p^{2}}\cosh pt  \label{I_0>chpt-a}
\end{equation}%
holds for $t>0$ with%
\begin{equation*}
\lambda _{0}=\frac{I_{0}\left( t_{0}\right) -1}{\left( \cosh pt_{0}-1\right)
/p^{2}},
\end{equation*}%
where $t_{0}$ is the unique solution of the equation%
\begin{equation*}
\frac{d}{dt}\frac{I_{0}\left( t\right) -1}{\cosh pt-1}=0
\end{equation*}%
on $\left( 0,\infty \right) $
\end{theorem}

\begin{proof}
(i) To prove the necessity for the inequalities (\ref{I_0-chpt}) to hold,
let us consider the ratio%
\begin{equation*}
R_{3}\left( t\right) =\frac{I_{0}\left( t\right) -1}{\left( \cosh
pt-1\right) /p^{2}}=\frac{\sum_{n=1}^{\infty }\frac{t^{2n}}{2^{2n}n!^{2}}}{%
\sum_{n=1}^{\infty }\frac{p^{2n-2}t^{2n}}{\left( 2n\right) !}}:=\frac{%
\sum_{n=1}^{\infty }\mu _{n}t^{2n}}{\sum_{n=1}^{\infty }\nu _{n}t^{2n}}.
\end{equation*}%
A simple computation yields%
\begin{eqnarray*}
\frac{\mu _{n+1}}{\nu _{n+1}}-\frac{\mu _{n}}{\nu _{n}} &=&\frac{1}{p^{2n}}%
\frac{\left( 2n+2\right) !}{2^{2n+2}\left( n+1\right) !^{2}}-\frac{1}{%
p^{2n-2}}\frac{\left( 2n\right) !}{2^{2n}n!^{2}} \\
&=&-\frac{1}{p^{2n}}\frac{\left( 2n\right) !}{2^{2n}n!^{2}}\left( p^{2}-%
\frac{2n+1}{2n+2}\right)  \\
&&\left\{ 
\begin{array}{cc}
\leq 0 & \text{if }p^{2}\geq \max_{n\in \mathbb{N}}\frac{2n+1}{2n+2}%
=1,\bigskip  \\ 
\geq 0 & \text{if }p^{2}\leq \min_{n\in \mathbb{N}}\frac{2n+1}{2n+2}=\frac{3%
}{4}.%
\end{array}%
\right. 
\end{eqnarray*}%
These show that the sequence $\{\mu _{n}/\nu _{n}\}$ is strictly decreasing
for $n\geq 1$ if $p\geq 1$ and increasing if $0<p\leq \sqrt{3}/2$, and so is 
$R_{3}$ by Lemma \ref{L-A/B}. Hence, we get that%
\begin{eqnarray*}
R_{3}\left( t\right)  &<&\lim_{t\rightarrow 0^{+}}R_{3}\left( t\right) =%
\frac{1}{2}\text{ if }p\geq 1, \\
R_{3}\left( t\right)  &>&\lim_{t\rightarrow 0^{+}}R_{3}\left( t\right) =%
\frac{1}{2}\text{ if }0<p\leq \sqrt{3}/2,
\end{eqnarray*}%
which proves (\ref{I_0-chpt}).

The necessity for the first inequality in (\ref{I_0-chpt}) to hold follows
from%
\begin{equation*}
\lim_{t\rightarrow 0^{+}}\frac{I_{0}\left( t\right) -\left( 1-\frac{1}{2p^{2}%
}+\frac{1}{2p^{2}}\cosh pt\right) }{t^{4}}=-\frac{1}{48}\left( p^{2}-\frac{3%
}{4}\right) \geq 0.
\end{equation*}

We prove the necessity for the second inequality in (\ref{I_0-chpt}) to hold
by proof by contradiction. Assume that there is a $q_{0}\in \left( \sqrt{3}%
/2,1\right) $ such that the second inequality in (\ref{I_0-chpt}) holds for $%
t>0$. Then there must be%
\begin{equation*}
\lim_{t\rightarrow \infty }\frac{I_{0}\left( t\right) -\left( 1-\frac{1}{%
2q_{0}^{2}}+\frac{1}{2q_{0}^{2}}\cosh q_{0}t\right) }{e^{q_{0}t}}\leq 0.
\end{equation*}%
Making use of the first inequality in (\ref{I-e^t}) leads to%
\begin{equation*}
\tfrac{I_{0}\left( t\right) -\left( 1-\frac{1}{2q_{0}^{2}}+\frac{1}{%
2q_{0}^{2}}\cosh q_{0}t\right) }{e^{q_{0}t}}>\tfrac{e^{\left( 1-q_{0}\right)
t}}{1+2t}-\left( 1-\tfrac{1}{2q_{0}^{2}}\right) e^{-q_{0}t}+\tfrac{1}{%
2q_{0}^{2}}\tfrac{1+e^{-2q_{0}t}}{2}\rightarrow \infty
\end{equation*}%
as $t\rightarrow \infty $, which gives a contradiction.

(ii) When $p\in \left( \sqrt{3}/2,1\right) $, it is clear that there exist a 
$n_{0}>1$ such that $\mu _{n+1}/\nu _{n+1}-\mu _{n}/\nu _{n}\leq 0$ for $%
1\leq n\leq n_{0}$ and $\mu _{n+1}/\nu _{n+1}-\mu _{n}/\nu _{n}\geq 0$ for $%
n\geq n_{0}$. By Lemma \ref{L-A/B-g} it follows that there is a $t_{0}>0$
such that $R_{3}$ is decreasing on $\left( 0,t_{0}\right) $ and increasing
on $\left( t_{0},\infty \right) $, and therefore, we have%
\begin{equation*}
R_{3}\left( t\right) \geq R_{3}\left( t_{0}\right) =\lambda _{0},
\end{equation*}%
that is, the inequality (\ref{I_0>chpt-a}).
\end{proof}

Similar to \cite[Remark 2.]{Yang-JFS-2015-370979} that the function%
\begin{equation*}
p\mapsto 1-\frac{1}{2p^{2}}+\frac{1}{2p^{2}}\cosh pt
\end{equation*}%
is increasing on $\left( 0,\infty \right) $. Letting $p=\sqrt{3}/2,3/4,1/%
\sqrt{2},2/3,1/2$ and $q=1$ in Theorem \ref{MT-TQ-A_p+G}, we have

\begin{corollary}
\label{C-TQ-A_p+G}It holds that%
\begin{eqnarray*}
\sqrt{\cosh t} &<&2\cosh \frac{t}{2}-1<\frac{9}{8}\cosh \frac{2}{3}-\frac{1}{%
8}<\cosh \frac{t}{\sqrt{2}}< \\
\frac{8}{9}\cosh \frac{3t}{4}+\frac{1}{9} &<&\frac{2}{3}\cosh \frac{\sqrt{3}t%
}{2}+\frac{1}{3}<I_{0}\left( t\right) <\frac{1+\cosh t}{2}
\end{eqnarray*}%
for $t>0$.
\end{corollary}

\begin{proof}
It remains to be proved that%
\begin{equation*}
\sqrt{\cosh t}<2\cosh \frac{t}{2}-1,
\end{equation*}%
which follows from%
\begin{equation*}
\left( 2\cosh \frac{t}{2}-1\right) ^{2}-\cosh t=2\left( \cosh \frac{t}{2}%
-1\right) ^{2}>0\text{.}
\end{equation*}
\end{proof}

\begin{theorem}
For $p>0$, the inequality%
\begin{equation}
I_{0}\left( t\right) >\left( \cosh pt\right) ^{1/\left( 2p^{2}\right) }
\label{lnI_0>lnchpt}
\end{equation}%
holds for $t>0$ if and only if $p\geq \sqrt{6}/4\approx 0.61237$. Its
reverse holds if and only if $p\in (0,1/2]$. In particular, we have%
\begin{equation}
\sqrt{\cosh t}<\cosh \frac{t}{\sqrt{2}}<\left( \cosh \frac{\sqrt{6}t}{4}%
\right) ^{4/3}<I_{0}\left( t\right) <\left( \cosh \frac{t}{2}\right)
^{2}<e^{t^{2}/4}  \label{Ch.-I-4}
\end{equation}
\end{theorem}

\begin{proof}
(i) The necessary condition for the inequality (\ref{lnI_0>lnchpt}) to hold
follows from%
\begin{equation*}
\lim_{t\rightarrow 0}\frac{I_{0}\left( t\right) -\left( \cosh pt\right)
^{1/\left( 2p^{2}\right) }}{t^{4}}=\frac{1}{24}\left( p^{2}-\frac{3}{8}%
\right) \geq 0.
\end{equation*}

Since the function $p\mapsto \left( \cosh pt\right) ^{1/\left( 2p^{2}\right)
}$ is strictly decreasing which is proved in \cite[Lemma 2]%
{Yang-JIA-2013-116}, to prove the sufficiency, it suffices to prove that the
inequality (\ref{lnI_0>lnchpt}) holds for $t>0$ when $p=\sqrt{6}/4$. In
fact, utilizing Theorem \ref{MT-TQ-A_p+G}, we only need to prove that%
\begin{equation*}
\frac{2}{3}\cosh \frac{\sqrt{3}t}{2}+\frac{1}{3}>\left( \cosh \frac{\sqrt{6}t%
}{4}\right) ^{4/3},
\end{equation*}%
which is equivalent to%
\begin{equation*}
f_{1}\left( x\right) :=\ln \left( \frac{2}{3}\cosh \left( \sqrt{2}x\right) +%
\frac{1}{3}\right) -\frac{4}{3}\ln \left( \cosh x\right) >0
\end{equation*}%
for $x>0$, where $x=\sqrt{6}t/4$.

Differentiation gives%
\begin{equation*}
f_{1}^{\prime }\left( x\right) =\frac{1}{3}\frac{f_{2}\left( x\right) }{%
\left( 2\cosh \left( \sqrt{2}x\right) +1\right) \cosh x},
\end{equation*}%
where%
\begin{equation*}
f_{2}\left( x\right) =6\sqrt{2}\sinh \left( \sqrt{2}x\right) \cosh x-8\cosh
\left( \sqrt{2}x\right) \sinh x-4\sinh x.
\end{equation*}%
Employing product into sum formula and Taylor expansion yields%
\begin{eqnarray*}
f_{2}\left( x\right)  &=&\left( 3\sqrt{2}-4\right) \sinh \left( \sqrt{2}%
+1\right) x+\left( 3\sqrt{2}+4\right) \sinh \left( \sqrt{2}-1\right)
x-4\sinh x \\
&=&\sum_{n=1}^{\infty }\frac{\left( 3\sqrt{2}-4\right) \left( \sqrt{2}%
+1\right) ^{2n-1}+\left( 3\sqrt{2}+4\right) \left( \sqrt{2}-1\right)
^{2n-1}-4}{\left( 2n-1\right) !}x^{2n-1} \\
&:&=\sum_{n=1}^{\infty }\frac{\xi _{n}}{\left( 2n-1\right) !}x^{2n-1}.
\end{eqnarray*}

Letting $\eta _{n}=\left( \sqrt{2}+1\right) ^{2n-1}$and noting that $\left( 
\sqrt{2}-1\right) ^{2n-1}=1/\eta _{n}$, we have%
\begin{eqnarray*}
\eta _{n}\xi _{n} &=&\left( 3\sqrt{2}-4\right) \eta _{n}^{2}-4\eta
_{n}+\left( 3\sqrt{2}+4\right)  \\
&=&\left( 3\sqrt{2}-4\right) \left( \eta _{n}-\sqrt{2}-1\right) \left( \eta
_{n}-5\sqrt{2}-7\right)  \\
&=&\left( 3\sqrt{2}-4\right) \left( \eta _{n}-\eta _{1}\right) \left( \eta
_{n}-\eta _{2}\right) .
\end{eqnarray*}%
It thus can be seen that $\xi _{1}=\xi _{2}=0$ and $\xi _{n}>0$ for $n\geq 3$
in view of $\eta _{n}>\eta _{2}>\eta _{1}$, which proves $f_{2}\left(
x\right) >0$. Hence, $f_{1}^{\prime }\left( x\right) >0$, and then $%
f_{1}\left( x\right) >f_{1}\left( 0\right) =0$ for $x>0$. Thus the
sufficiency follows.

(ii) The sufficiency follows from the last inequality in Corollary \ref%
{C-TQ-A_p+G} with the decreasing property of the function $p\mapsto \left(
\cosh pt\right) ^{1/\left( 2p^{2}\right) }$ on $\left( 0,\infty \right) $.
It remains to be proved the necessity. If there is a $p_{0}\in \left( 1/2,%
\sqrt{6}/4\right) $ such that $I_{0}\left( t\right) <\left( \cosh
p_{0}t\right) ^{1/\left( 2p_{0}^{2}\right) }$ for $t>0$, then there must be%
\begin{equation*}
\lim_{t\rightarrow \infty }\frac{I_{0}\left( t\right) -\left( \cosh
p_{0}t\right) ^{1/\left( 2p_{0}^{2}\right) }}{e^{t/\left( 2p_{0}\right) }}%
\leq 0.
\end{equation*}%
But by the first inequality in (\ref{I-e^t}), we have%
\begin{equation*}
\frac{I_{0}\left( t\right) -\left( \cosh p_{0}t\right) ^{1/\left(
2p_{0}^{2}\right) }}{e^{t/\left( 2p_{0}\right) }}>\tfrac{1}{1+2t}\frac{e^{t}%
}{e^{t/\left( 2p_{0}\right) }}-\left( \frac{1+e^{-2p_{0}t}}{2}\right)
^{1/\left( 2p_{0}^{2}\right) }\rightarrow \infty 
\end{equation*}%
as $t\rightarrow \infty $, which yields a contradiction.

Taking $p=1,1/\sqrt{2},\sqrt{6}/4,1/2,0^{+}$ gives the chain of inequalities
(\ref{Ch.-I-4}).

The theorem is proved.
\end{proof}

\begin{theorem}
Let $\theta \in \left[ 0,\pi /2\right] $. Then the inequality%
\begin{equation}
I_{0}\left( t\right) >\frac{\cosh \left( t\cos \theta \right) +\cosh \left(
t\sin \theta \right) }{2}  \label{I_0-coshttr.}
\end{equation}%
holds for $t>0$ if and only if $\theta \in \left[ \pi /8,3\pi /8\right] $.
In particular, it holds that%
\begin{equation}
I_{0}\left( t\right) >\frac{1}{2}\left( \cosh \tfrac{\sqrt{2-\sqrt{2}}t}{2}%
+\cosh \tfrac{\sqrt{2+\sqrt{2}}t}{2}\right) >\frac{1}{2}\left( \cosh \tfrac{%
\sqrt{3}t}{2}+\cosh \tfrac{t}{2}\right) >\cosh \frac{t}{\sqrt{2}}
\label{I_0-coshttr.-ch}
\end{equation}%
for $t>0$.
\end{theorem}

\begin{proof}
The necessity can follow from%
\begin{equation*}
\lim_{t\rightarrow 0}\frac{I_{0}\left( t\right) -\frac{\cosh \left( t\cos
\theta \right) +\cosh \left( t\sin \theta \right) }{2}}{t^{4}}=-\frac{1}{192}%
\cos 4\theta \leq 0,
\end{equation*}%
which yields $4\theta \in \left[ \pi /2,3\pi /2\right] $, that is, $\theta
\in \left[ \pi /8,3\pi /8\right] $.

Differentiation gives%
\begin{eqnarray*}
\frac{\partial }{\partial \theta }\frac{\cosh \left( t\cos \theta \right)
+\cosh \left( t\sin \theta \right) }{2} &=&\frac{t^{2}\sin 2\theta }{4}%
\left( \frac{\sinh \left( t\sin \theta \right) }{t\sin \theta }-\frac{\sinh
\left( t\cos \theta \right) }{t\cos \theta }\right) , \\
\frac{d}{dx}\frac{\sinh x}{x} &=&\frac{1}{x}\left( \cosh x-\frac{\sinh x}{x}%
\right) >0,
\end{eqnarray*}%
which show that the function $\theta \mapsto \left[ \cosh \left( t\cos
\theta \right) +\cosh \left( t\sin \theta \right) \right] /2$ is decreasing
on $\left[ 0,\pi /4\right] $ and increasing on $\left[ \pi /4,\pi /2\right] $%
. Thus, to prove the sufficiency, it is enough to prove the inequality (\ref%
{I_0-coshttr.}) holds when $\theta =\pi /8$.

Expanding in power series yields%
\begin{equation*}
R_{4}\left( t\right) :=\frac{\cosh \left( t\cos \theta \right) +\cosh \left(
t\sin \theta \right) }{2I_{0}\left( t\right) }=\frac{\sum_{n=0}^{\infty }%
\frac{\left( \frac{2-\sqrt{2}}{4}\right) ^{n}+\left( \frac{2+\sqrt{2}}{4}%
\right) ^{n}}{\left( 2n\right) !}t^{2n}}{\sum_{n=0}^{\infty }\frac{2}{%
2^{2n}n!^{2}}t^{2n}}:=\frac{\sum_{n=0}^{\infty }\rho _{n}t^{2n}}{%
\sum_{n=0}^{\infty }\sigma _{n}t^{2n}}.
\end{equation*}%
Straightforward computations lead to%
\begin{equation*}
\frac{\rho _{n}}{\sigma _{n}}=\frac{1}{2}\frac{n!^{2}\left( \sqrt{2}\right)
^{n}}{\left( 2n\right) !}\left( \left( \sqrt{2}-1\right) ^{n}+\left( \sqrt{2}%
+1\right) ^{n}\right) ,
\end{equation*}%
\begin{eqnarray*}
\frac{\rho _{n+1}}{\sigma _{n+1}}\left/ \frac{\rho _{n}}{\sigma _{n}}\right.
&=&\frac{\sqrt{2}}{2}\frac{n+1}{2n+1}\frac{\left( \sqrt{2}-1\right)
^{n+1}+\left( \sqrt{2}+1\right) ^{n+1}}{\left( \sqrt{2}-1\right) ^{n}+\left( 
\sqrt{2}+1\right) ^{n}}, \\
\frac{\rho _{n+1}}{\sigma _{n+1}}\left/ \frac{\rho _{n}}{\sigma _{n}}\right.
-1 &=&-\tfrac{\sqrt{2}}{2\left( 2n+1\right) }\tfrac{\left( n+\sqrt{2}%
-1\right) \left( \sqrt{2}-1\right) ^{n-1}+\left( n-\sqrt{2}-1\right) \left( 
\sqrt{2}+1\right) ^{n-1}}{\left( \sqrt{2}-1\right) ^{n}+\left( \sqrt{2}%
+1\right) ^{n}}<0
\end{eqnarray*}%
for $n\geq 0$. This means that the sequence $\{\rho _{n}/\sigma _{n}\}$ is
decreasing for $n\geq 0$, and so\ is $R_{4}$, which proves the sufficiency.

Applying the decreasing property of the function $\theta \mapsto \left[
\cosh \left( t\cos \theta \right) +\cosh \left( t\sin \theta \right) \right]
/2$ and letting $\theta =\pi /8$, $\pi /6$, $\pi /4$ show the inequalities (%
\ref{I_0-coshttr.-ch}).

The proof of this theorem ends.
\end{proof}

\section{Remarks}

\begin{remark}
In the proof of Property \ref{P-I<e^t/s}, we in fact give a simple proof of
the inequality $I_{0}\left( t\right) >\left( \sinh t\right) /t$, which is
equivalent to the first inequality in (\ref{Qi-I-1}). Furthermore, we have%
\begin{equation}
I_{0}\left( t\right) >\frac{\sinh t}{t}+\frac{3\left( 4-\pi \right) }{\pi }%
\frac{\left( t\sinh t-2\cosh t+2\right) }{t^{2}}.  \label{I_0>Y}
\end{equation}%
Indeed, Lupas \cite{Lupas-UBPEFSM-17-19-1972} has proven that%
\begin{equation*}
T\left( f,g\right) \geq \frac{12}{\left( b-a\right) ^{4}}\left(
\int_{a}^{b}\left( x-\frac{a+b}{2}\right) f\left( x\right) dx\right) \left(
\int_{a}^{b}\left( x-\frac{a+b}{2}\right) g\left( x\right) dx\right)
\end{equation*}%
if both $f,g$ are convex on interval $\left[ a,b\right] $, where $T\left(
f,g\right) $ is the Chebyshev functional defined by%
\begin{equation*}
T\left( f,g\right) =\frac{1}{b-a}\int_{a}^{b}f\left( x\right) g\left(
x\right) dx-\frac{1}{\left( b-a\right) ^{2}}\left( \int_{a}^{b}f\left(
x\right) dx\right) \left( \int_{a}^{b}g\left( x\right) dx\right) .
\end{equation*}%
Differentiation yields%
\begin{equation*}
\left( \frac{1}{\sqrt{1-x^{2}}}\right) ^{\prime \prime }=\frac{2x^{2}+1}{%
\left( 1-x^{2}\right) ^{5/2}}>0\text{ \ and \ }\frac{\partial ^{2}}{\partial
x^{2}}\cosh \left( tx\right) =t^{2}\cosh tx>0.
\end{equation*}%
Then by Lupas's inequality we have%
\begin{eqnarray*}
&&\int_{0}^{1}\frac{\cosh \left( tx\right) }{\sqrt{1-x^{2}}}dx-\left(
\int_{0}^{1}\frac{dx}{\sqrt{1-x^{2}}}\right) \left( \int_{0}^{1}\cosh \left(
tx\right) dx\right) \\
&>&12\left( \int_{0}^{1}\frac{x-1/2}{\sqrt{1-x^{2}}}dx\right) \left(
\int_{0}^{1}\left( x-1/2\right) \cosh \left( tx\right) dx\right) \\
&=&\frac{3\left( 4-\pi \right) }{2}\frac{\left( t\sinh t-2\cosh t+2\right) }{%
t^{2}},
\end{eqnarray*}%
which together with (\ref{I_0-Yang-2}) implies (\ref{I_0>Y}).

When $0<t<0.8305...$, the lower bound given in (\ref{I_0>Y}) is weaker than
one in \cite[(1.5)]{Guo-RG-}; when $t>0.8305...$, the lower bound given in (%
\ref{I_0>Y}) is better than one in \cite[(1.5)]{Guo-RG-}.
\end{remark}

We recall the definition of "power-type mean". Let $p\in \mathbb{R}$ and $M$
be a bivariate mean. Then the function $M_{p}:\left( 0,\infty \right) \times
\left( 0,\infty \right) \rightarrow \left( 0,\infty \right) $ defined by%
\begin{equation}
M_{p}\equiv M_{p}(a,b)=M(a^{p},b^{p})^{1/p}\text{ if }p\neq 0\text{ and }%
M_{0}=\sqrt{ab}  \label{M_p...}
\end{equation}%
is proved to be a mean (see \cite[Theorem 1]{Yang-JMI-9(2)-2015}), and is
called \textquotedblleft $p$-order $M$ mean\textquotedblright\ or
\textquotedblleft power-type mean\textquotedblright . Also, we have%
\begin{equation}
M_{p\lambda }\left( a,b\right) =M\left( a^{p\lambda },b^{p\lambda }\right)
^{1/\left( p\lambda \right) }=M_{p}\left( a^{\lambda },b^{\lambda }\right)
^{1/\lambda }  \label{M_pl}
\end{equation}%
for all $\lambda \in \mathbb{R}$.

\begin{remark}
The first inequality in (\ref{Qi-I-2}) can be written as $TQ\left(
a,b\right) <A_{1/2}\left( a,b\right) $. Also, it has been proven that $%
A_{2/3}\left( a,b\right) $ is the best lower bound for identric
(exponential) mean $\mathcal{I}\left( a,b\right) $ (see \cite%
{Stolarsky-AMM-87-1980}, \cite{Pittinger-UBPEFSMF-680-1980}). Then by taking 
$M=A$, $p=2/3$, $\lambda =3/4$ in identity (\ref{M_pl}) we have%
\begin{equation*}
TQ\left( a,b\right) <A_{1/2}\left( a,b\right) =A_{2/3}\left(
a^{3/4},b^{3/4}\right) ^{4/3}<\mathcal{I}\left( a^{3/4},b^{3/4}\right)
^{4/3}=\mathcal{I}_{3/4}\left( a,b\right) ,
\end{equation*}%
which is superior to the second inequality in (\ref{Qi-I-1}), that is, 
\begin{equation*}
TQ\left( a,b\right) <\mathcal{I}_{3/4}\left( a,b\right) <\mathcal{I}\left(
a,b\right) ,
\end{equation*}%
because that the $p$-order identric (exponential) mean is increasing in $%
p\in \mathbb{R}$ (see \cite{Yang-JMI-9(2)-2015}).

Further, expanding in power series gives%
\begin{equation*}
I_{0}\left( t\right) -\exp \left( \frac{t}{\tanh pt}-1/p\right) =-\frac{1}{3}%
\left( p-\frac{3}{4}\right) t^{2}+O\left( t^{4}\right) ,
\end{equation*}%
which implies that the condition $p\geq 3/4$ is necessary for the inequality 
$I_{0}\left( t\right) <\exp \left( t\coth pt-1/p\right) $ to hold for all $%
t>0$. This statement can be stated as a theorem.
\end{remark}

\begin{theorem}
For $a,b>0$ with $a\neq b$, the inequality%
\begin{equation*}
TQ\left( a,b\right) <\mathcal{I}_{p}\left( a,b\right)
\end{equation*}%
holds if and only if $p\geq 3/4$.
\end{theorem}

\begin{remark}
For the Toader mean $\mathcal{T}\left( a,b\right) $ of positive numbers $a$
and $b$ defined by (\ref{Toadermean}), it was proved in \cite%
{Alzer-JCAM-172-2004,Qiu-JHIEE-20(1)-2000}, \cite{Qiu-JHIEE-3-1997,
Barnard-JMA-31-2000} that%
\begin{equation}
A_{3/2}\left( a,b\right) <\mathcal{T}\left( a,b\right) <A_{\ln 2/\ln \left(
\pi /2\right) }\left( a,b\right)  \label{T-A_p}
\end{equation}%
hold for $a,b>0$ with $a\neq b$, where $2/3$ and $\ln 2/\ln \left( \pi
/2\right) $ are the best possible. Very recently, we have shown that%
\begin{equation}
\mathcal{T}\left( a,b\right) <\mathcal{I}_{9/4}\left( a,b\right) .
\label{T-I_p}
\end{equation}

Replacing $\left( a,b\right) $ by $\left( a^{1/3}.b^{1/3}\right) $ in the
first inequality of (\ref{T-A_p}) and (\ref{T-I_p}), we get that%
\begin{equation*}
A_{1/2}\left( a,b\right) ^{1/3}<\mathcal{T}\left( a^{1/3},b^{1/3}\right) =%
\mathcal{T}_{1/3}\left( a,b\right) ^{1/3}<\mathcal{I}_{3/4}\left( a,b\right)
^{1/3},
\end{equation*}%
which can be simplified as%
\begin{equation}
A_{1/2}\left( a,b\right) <\mathcal{T}_{1/3}\left( a.b\right) <\mathcal{I}%
_{3/4}\left( a,b\right) .  \label{A-T-I}
\end{equation}%
This together with the inequality $TQ\left( a,b\right) <A_{1/2}\left(
a,b\right) $ gives a nice chain of inequalities:%
\begin{equation}
TQ\left( a,b\right) <A_{1/2}\left( a,b\right) <\mathcal{T}_{1/3}\left(
a.b\right) <\mathcal{I}_{3/4}\left( a,b\right) .  \label{TQ-A-T-I}
\end{equation}
\end{remark}

\begin{remark}
Theorem \ref{MT-TQ-LA} shows that the double inequality%
\begin{equation}
L\left( a,b\right) ^{3/4}A\left( a,b\right) ^{1/4}<TQ\left( a,b\right) <%
\frac{3}{4}L\left( a,b\right) +\frac{1}{4}A\left( a,b\right)
\label{I_0-L-A-m}
\end{equation}%
holds for all $a,b>0$ with $a\neq b$, where the exponents $3/4$, $1/4$ and
weights $3/4$, $1/4$ are the best. Using the well-known inequalities%
\begin{equation*}
L\left( a,b\right) <\frac{A\left( a,b\right) +2G\left( a,b\right) }{3}
\end{equation*}%
proved in \cite{Carlson-AMM-79-1972} and $A\left( a,b\right) >L\left(
a,b\right) $, we have%
\begin{equation*}
L\left( a,b\right) <L\left( a,b\right) ^{3/4}A\left( a,b\right)
^{1/4}<TQ\left( a,b\right) <\frac{3}{4}L\left( a,b\right) +\frac{1}{4}%
A\left( a,b\right) <\frac{A\left( a,b\right) +G\left( a,b\right) }{2}.
\end{equation*}%
It is thus clear that our inequalities (\ref{I_0-L-A-m}) improve Qi et al.'s
results (\ref{Qi-I-1}) and (\ref{Qi-I-2}).
\end{remark}

\begin{remark}
For the Gauss compound mean $AGM\left( a,b\right) $ of positive numbers $a$
and $b$ defined by (\ref{AGM}), it has been proven in \cite%
{Carlson-SIAMR-33-1991}, \cite{Borwein-JMAA-177(2)-1993}, \cite[Theorem 1]%
{Yang-JIA-192-2014} that%
\begin{equation}
L\left( a,b\right) <AGM\left( a,b\right) <L\left( a,b\right) ^{3/4}A\left(
a,b\right) ^{1/4}<L_{3/2}\left( a,b\right)   \label{AGM-L-A}
\end{equation}%
for $a,b>0$ with $a\neq b$, where $L_{p}\left( a,b\right) =L\left(
a^{p},b^{p}\right) ^{1/p}$ is the $p$-order logarithmic mean. This in
combination with our inequalities (\ref{I_0-L-A-m}) and (\ref{TQ-A-T-I})
yields a more nice chain of inequalities involving Gauss compound mean,
Toader mean and Toader-Qi mean:%
\begin{equation}
\begin{array}{l}
L\left( a,b\right) <AGM\left( a,b\right) <L\left( a,b\right) ^{3/4}A\left(
a,b\right) ^{1/4}<TQ\left( a,b\right) \bigskip  \\ 
\text{ \ \ \ \ }<\frac{3}{4}L\left( a,b\right) +\frac{1}{4}A\left(
a,b\right) <A_{1/2}\left( a,b\right) <\mathcal{T}_{1/3}\left( a.b\right) <%
\mathcal{I}_{3/4}\left( a,b\right) .%
\end{array}
\label{AGM-TQi-T_p}
\end{equation}

Moreover, inspired by the third inequality in \ref{AGM-L-A} and the first
inequality in \ref{I_0-L-A-m}, we propose a conjecture as follows.
\end{remark}

\begin{conjecture}
For $a,b>0$ with $a\neq b$, the inequality%
\begin{equation*}
TQ\left( a,b\right) >L_{3/2}\left( a,b\right)
\end{equation*}%
holds.
\end{conjecture}

\begin{remark}
From Corollary \ref{C-TQ-A_p+G} or the chain of inequalities (\ref{Ch.-I-4})
we have%
\begin{equation*}
\sqrt{\cosh t}<I_{0}\left( t\right) <\frac{\cosh t+1}{2},
\end{equation*}%
which is equivalent to%
\begin{equation*}
\sqrt{A\left( a,b\right) G\left( a,b\right) }<TQ\left( a,b\right) <\frac{%
A\left( a,b\right) +G\left( a,b\right) }{2}.
\end{equation*}%
This in conjunction with the inequalities%
\begin{equation*}
\sqrt{A\left( a,b\right) G\left( a,b\right) }<\sqrt{L\left( a,b\right) 
\mathcal{I}\left( a,b\right) }<\frac{L\left( a,b\right) +\mathcal{I}\left(
a,b\right) }{2}<\frac{A\left( a,b\right) +G\left( a,b\right) }{2}
\end{equation*}%
proved in \cite{Alzer-AM-47-1986} gives rise to another conjecture.
\end{remark}

\begin{conjecture}
For $a,b>0$ with $a\neq b$, the inequalities%
\begin{equation*}
\sqrt{A\left( a,b\right) G\left( a,b\right) }<TQ\left( a,b\right) <\sqrt{%
L\left( a,b\right) \mathcal{I}\left( a,b\right) }<\frac{L\left( a,b\right) +%
\mathcal{I}\left( a,b\right) }{2}<\frac{A\left( a,b\right) +G\left(
a,b\right) }{2}
\end{equation*}%
hold.
\end{conjecture}

\begin{remark}
Kazarinoff in \cite{Kazarinoff-MS-1956} gave the following Wallis
inequalities:%
\begin{equation}
\frac{1}{\sqrt{\pi \left( n+1/2\right) }}<\frac{\left( 2n-1\right) !!}{%
\left( 2n\right) !}<\frac{1}{\sqrt{\pi \left( n+1/4\right) }}\text{, \ }n\in 
\mathbb{N}\text{.}  \label{W-KI}
\end{equation}%
For more information on the Wallis inequalities, please refer to \cite%
{Qi-JIA-2010-493058} and the references therein. Our Lemma \ref{L-s_n} tells
us that the sequence $\{s_{n}\}$ is strictly decreasing for $n\in \mathbb{N}$%
, therefore, we have%
\begin{equation*}
\frac{2}{\pi }<s_{n}=\left( 2n+1\right) \left( \frac{\left( 2n-1\right) !!}{%
2^{n}n!}\right) ^{2}<\frac{3}{4},
\end{equation*}%
which is equivalent to%
\begin{equation*}
\frac{1}{\sqrt{\pi }\sqrt{n+1/2}}<\frac{\left( 2n-1\right) !!}{2^{n}n!}<%
\frac{\sqrt{6}}{4\sqrt{n+1/2}}.
\end{equation*}

From the proof of Theorem \ref{MT-TQ-sqrLA-w} it is seen that the sequence $%
\{c_{n}/d_{n}\}$ defined by%
\begin{equation*}
\frac{c_{n}}{d_{n}}=\frac{2^{2n}s_{n}-1}{2^{2n}-1}
\end{equation*}%
is strictly increasing for $n=1,2$ and decreasing for $n\geq 2$. Therefore,
we have%
\begin{equation*}
\min \left( \frac{2}{3},\frac{2}{\pi }\right) =\min \left( \frac{c_{1}}{d_{1}%
},\lim_{n\rightarrow \infty }\frac{c_{n}}{d_{n}}\right) <\frac{2^{2n}s_{n}-1%
}{2^{2n}-1}\leq \frac{c_{2}}{d_{2}}=\frac{41}{60},
\end{equation*}%
which is equivalent to%
\begin{equation}
\sqrt{\frac{\left( \pi -2\right) 2^{-2n}+2}{\pi \left( 2n+1\right) }}<\frac{%
\left( 2n-1\right) !!}{\left( 2n\right) !!}<\sqrt{\frac{41+19\times 2^{-2n}}{%
60\left( 2n+1\right) }},n\in \mathbb{N}\text{.}  \label{W-YI}
\end{equation}%
This in fact gives a new Wallis type inequality, and the lower bound given
in (\ref{W-YI}) is clearly superior to the one given in (\ref{W-KI}).
\end{remark}

\end{document}